\documentclass[a4paper,12pt]{article}

\usepackage[T1]{fontenc}
\usepackage[utf8]{inputenc}
\usepackage{graphicx}

\usepackage{amsmath}
\usepackage{amsfonts}
\usepackage{amssymb}
\usepackage{amsthm}
\usepackage{array}
\usepackage{color}
\usepackage{enumerate}
\usepackage{epstopdf}

\usepackage{caption}
\usepackage{tikz}
\usetikzlibrary{quotes,angles,positioning}
\usepackage{tkz-berge}
\usepackage{tkz-euclide} 
\usepackage{color}
\usepackage{xcolor}

\newtheorem{theorem}{Theorem}[section]

\newtheorem{proposition}[theorem]{Proposition}
\newtheorem{lemma}[theorem]{Lemma}

\newtheorem{remark}[theorem]{Remark}

\numberwithin{equation}{section}

\newcommand{\Z}{\mathbb{Z}}

\newcommand{\R}{\mathbb{R}}
\newcommand{\C}{\mathbb{C}}

\newcommand{\Lbb}{\mathbb{L}}

\newcommand{\Sbb}{\mathbb{S}}
\newcommand{\Tbb}{\mathbb{T}}

\newcommand{\Bcal}{\mathcal{B}}
\newcommand{\Ccal}{\mathcal{C}}

\newcommand{\Wcal}{\mathcal{W}}

\newcommand{\norm}[2]{\left\| #1 \right\|_{#2}}
\newcommand{\dd}{\;{\rm d}}
\newcommand{\de}{{\rm d}}

\newcommand{\acts}{\curvearrowright}

\DeclareMathOperator{\Vol}{Vol}

\DeclareMathOperator{\Ker}{Ker}

\DeclareMathOperator{\Sp}{Sp}

\DeclareMathOperator{\ess}{ess}

\usepackage[top=3cm, bottom=2cm, left=1.5cm, right=1.5cm]{geometry}

\title{Sinai billiard maps with Ruelle resonances}
\author{Damien \textsc{Thomine}}
\date{}

\begin{document}

\maketitle

\begin{abstract}

We construct families of two-dimensional Sinai billiards whose transfer operators 
have Ruelle resonances arbitrarily close to $1$. Our method involves taking a large 
enough cover of an initial billiard table, and relating the transfer operator of 
the covering table to twisted transfer operators of the initial table. We also 
study the distribution of these resonances which are close to $1$.
\end{abstract}

Convex billiards tables are one of the classical models of chaotic dynamics, 
dating back to Sinai~\cite{Sinai:1970}. Over the years, many of their statistical 
properties have been proved, starting with their ergodicity~\cite{Sinai:1970}, 
and up to the Central Limit Theorem, the exponential decay of correlations~\cite{Young:1998, Chernov:1999} 
and large deviations for the collision map~\cite{ReyBelletYoung:2008,MelbourneNicol:2008}.

\smallskip

In the last few years, the approach \textit{via} the study of the spectral properties 
of the transfer operator bore fruits, with M.~Demers and H.-K.~Zhang constructing 
Banach spaces $\Bcal$ on which the transfer operator acts quasi-compactly~\cite{DemersZhang:2012, DemersZhang:2014}. 
This implies the previous results, and led to a finer understanding of the statistical 
properties of the billiard flow using Dolgopyat-type arguments~\cite{BaladiDemersLiverani:2018}.

\smallskip

Since the transfer operator acting on a suitable Banach space is quasi-compact, 
one can define Ruelle resonances, that is, eigenvalues of the transfer operator. 
There is at least one such eigenvalue ($1$, corresponding to constant functions). 
A question, asked by V.~Baladi, was whether one could find billiard tables 
with non-trivial Ruelle resonances.

\smallskip

There are relatively few examples for which we are able to describe explicitly 
the spectrum of the transfer operator~; one such instance is given by the work 
of O.~Bandtlow, W.~Just and J.~Slipantschuk on Blaschke products~\cite{BandtlowJustSplipantschuk:2017a, BandtlowJustSplipantschuk:2017b}. 
A generic Anosov diffeomorphism of the $2$-torus also admits non-trivial 
Ruelle resonances~\cite{Adam:2017}.

\smallskip

We prove that, for a suitable choice of billiard tables, the transfer operator admits non-trivial Ruelle resonances. 
In a nutshell, we fix an initial billiard table, and relate the transfer operator on Abelian covers of the billiard table 
to twisted transfer operators for this initial table. Then the eigenvalues of a twisted transfer operator 
appear as Ruelle resonances of the transfer operator on a corresponding cover. A perturbative argument finally shows 
that, if the cover is large enough, there must exist such resonances close to $1$ (Theorem~\ref{thm:ExistenceResonances}). 

\smallskip

While we found it independently, our method is very close to the one used by D.~Jakobson, F.~Naud and L.~Soares~\cite{JakobsonNaudSoares:2019} 
to prove the existence of Ruelle resonances for geodesic flows on convex-cocompact surfaces of constant negative curvature. 
The main difference is that, for the geodesic flow in constant curvature, an approach 
\textit{via} dynamical zeta functions is available, which simplifies many arguments (in particular those 
relying on the time-reversal symmetry). This approach is not available in the context of billiards, 
so we provide more elementary, and more robust, arguments.

\smallskip

As in~\cite{JakobsonNaudSoares:2019}, we also show that those resonances which are close to $1$ are real, 
and study their distribution for families of large covers (Propositions~\ref{prop:DistributionResonancesDim1} 
and~\ref{prop:DistributionResonancesDim2}).

\smallskip

While we expect this method to work as well with the billiard flow, some groundwork 
is necessary to be able to deal with the Banach spaces constructed in~\cite{BaladiDemersLiverani:2018}. 
As a consequence, we only discuss the collision map.

\smallskip

The necessary background and our results are exposed in Section~\ref{sec:ContexteResultats}.
We prove the existence of resonances in Section~\ref{sec:ExistenceResonance}, and study their 
distribution in Section~\ref{sec:DistributionResonances}.

\section*{Acknowledgements}

The author would like to thank Viviane Baladi for suggesting the problem, and to Viviane Baladi, 
Mark Demers, S\'ebastien Gou\"ezel and Fr\'ed\'eric Naud for their remarks.

\smallskip

The Oberwolfach seminar ``Anisotropic Spaces and their Applications to Hyperbolic and Parabolic Systems'' 
held in June 2019 greatly helped this research.

\section{Context and results}
\label{sec:ContexteResultats}

\subsection{Sinai billiards and Ruelle spectrum}
\label{subsec:BillardSinai}

A planar Sinai billiard with finite horizon is given by a finite number of 
non-overlapping closed convex regions $(\Gamma_k)_{1 \leq k \leq d}$ of the torus $\Tbb^2$, 
whose boundaries $(\gamma_k)_{1 \leq k \leq d}$ are $\Ccal^3$ with non-vanishing curvature, 
and such that any line in the torus meets the interior of one of the $\Gamma_k$'s 
(the so-called finite horizon condition).

\smallskip

The billiard table is $Q = \Tbb^2 \setminus \bigcup_{k=1}^d \mathring{\Gamma_k}$. We consider 
the dynamics of a point particle moving at unit speed in $Q$, with specular reflection at the 
obstacles. The state space for this flow is three-dimensional, and there is a natural Poincar\'e section: 
the set $M := \bigcup_{k=1}^d \gamma_k \times [-\pi/2, \pi/2]$ of outward-facing unit vectors at the boundaries 
of the obstacles. The finite horizon condition implies that the return time to this Poincar\'e section 
is bounded. Let $T : M \to M$ be the first return map to $M$, which is also called the collision map.

\begin{figure}[!h]
  \centering
  \scalebox{0.5}{
  \begin{tikzpicture}
   \draw [dashed] (-5,-5) -- (5,-5) -- (5,5) -- (-5,5) -- (-5,-5) ;
   \draw [fill = black, opacity=0.5](0,0) circle (4) ;
   \draw [fill = black, opacity=0.5] (-5,-5) -- (-2.5,-5) arc (0:90:2.5) -- cycle;
   \draw [fill = black, opacity=0.5] (5,-5) -- (5,-2.5) arc (90:180:2.5) -- cycle;
   \draw [fill = black, opacity=0.5] (5,5) -- (2.5,5) arc (180:270:2.5) -- cycle;
   \draw [fill = black, opacity=0.5] (-5,5) -- (-5,2.5) arc (270:360:2.5) -- cycle;
   \draw [dotted] (1.8,-3.8) -- (2.57,-4.4) ;
   \draw (2.57,-4.4) -- (-2.52,-4.7) -- (1,-3.87) -- (3,-3.5) -- (3,-2.65) -- (4.3,-2.6) -- (5,-2) ;
   \draw (-5,-2) -- (-3.9,-0.89) ;
   \draw [dotted] (-3.9,-0.89) -- (-5,-0.5) ;
   \draw [->, line width=1mm] (2.57,-4.4) -- (1.87,-4.45) ;
   \draw [->, line width=1mm] (-2.52,-4.7) -- (-1.84,-4.54) ;
   \draw [->, line width=1mm] (1,-3.87) -- (1.69,-3.74) ;
   \draw [->, line width=1mm] (3,-3.5) -- (3,-2.8) ;
   \draw [->, line width=1mm] (3,-2.65) -- (3.7,-2.62) ;
   \draw [->, line width=1mm] (4.3,-2.6) -- (4.83,-2.14) ;
  \end{tikzpicture}
  }
  \caption{A trajectory in a finite horizon Sinai billiard table, with a marked outward-facing unit vector at each collision.}
\end{figure}
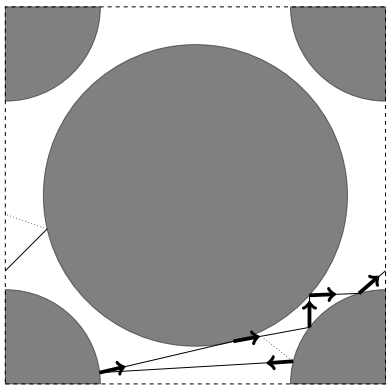

The map $T$ is hyperbolic with codimension $1$ singularities (which correspond to grazing 
trajectories on the billiard table), and a finite number of domains of continuity. 
It preserves a symplectic form on $M$, and thus the associated 
Liouville measure $\mu = \frac{\cos (\theta)}{2 \sum_{k=1}^d |\gamma_k|} \de \ell \de \theta$ on $M$. 

\smallskip

The collision map $T$ is one of the most prominent examples of 
chaotic maps ; for instance, it exhibits exponential decay of correlations against H\"older observables. 
More precisely, for any $\eta >0$, there exist constants $C>0$ and $\lambda >1$ such that
\begin{equation}
 \label{eq:DecayCorrelations}
 \left| \int_M \varphi \cdot \psi \circ T^n \dd \mu - \int_M \varphi \dd \mu \int_M \psi \dd \mu \right| 
 \leq C \lambda^{-n} \norm{\varphi}{\eta} \norm{\psi}{\eta}, 
\end{equation}
where $\norm{\cdot}{\eta}$ is the $\eta$-H\"older norm. Equation~\eqref{eq:DecayCorrelations} 
can be seen as a consequence of spectral properties of the composition operator 
$h \mapsto h \circ T$, or its dual, the transfer operator $P$. The transfer operator $P$ 
acts on $\Lbb^1 (M, \mu)$ by:
\begin{equation*}
 P (h) 
 = h \circ T^{-1} \quad \forall h \in \Lbb^1 (M, \mu).
\end{equation*}
The operator $P$ also acts on various function or distribution spaces; of interest to us will be the Banach spaces 
of anisotropic distributions $\Bcal$ constructed by M.~Demers and H.-K.~Zhang~\cite{DemersZhang:2012, DemersZhang:2014}. 

\smallskip

The anisotropic distributions on $M$ constructed by M.~Demers and H.-K.~Zhang are regular in the unstable direction, 
and dual of regular in the stable direction. By~\cite[Lemma~2.1]{DemersZhang:2012}, 
\begin{equation*}
 \Ccal^{1/3} (M) 
 \to \Bcal 
 \to \Ccal^{1/3} (T^{-n} \Wcal^s)^*,
\end{equation*}
where $\Ccal^\gamma$ is the space of $\gamma$-H\"older functions, $\Wcal^s$ is a space of stable curves, and the 
inclusions are continuous and injective. In addition, the injection $\Ccal^1 (M) \subset \Bcal$ has a dense image.

\smallskip

The operator $P$ acts continuously on $\Bcal$. By~\cite[Proposition~2.3]{DemersZhang:2012}, its action is quasi-compact: 
$1$ belongs to the spectrum of $P$ (since $\mathbf{1} \in \Bcal$ and $P (\mathbf{1}) = \mathbf{1}$), and the essential 
spectral radius $\rho_{\ess} (P \acts \Bcal)$ of $P$ acting on $\Bcal$ is strictly smaller than $1$. More precisely, 
there exists a constant $\rho_0 >0$, depending only on the minimal travel time between obstacles and the minimal curvature 
of the obstacles, such that $\rho_{\ess} (P \acts \Bcal) \leq \rho_0 <1$. 
The spectrum of $P$ in $\overline{B} (0, \rho_0)^c$ is discrete, contained in $\overline{B} (0, 1)$, 
and consists in (at most) countably many eigenvalues of finite multiplicity. Such eigenvalues are 
\textit{Ruelle resonances} of the transfer operator. 

\smallskip

In addition, when acting on $\Bcal$, the operator $P$ has a spectral gap: its resonance $1$ is simple 
and is the only resonance of modulus $1$. Hence, $P$ is the sum of the rank $1$ 
projection $h \mapsto \int_M h \dd \mu \cdot \mathbf{1}$ and of an operator of spectral radius 
strictly smaller than $1$. The exponential decay of correlations~\eqref{eq:DecayCorrelations} follows.

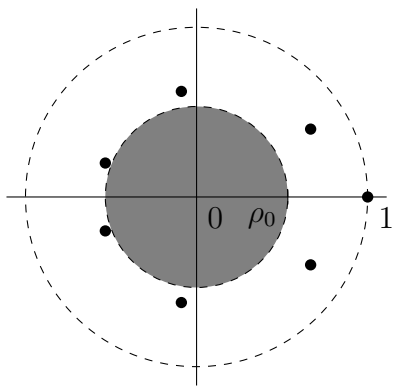
\begin{figure}[!h]
  \centering
  \scalebox{1}{
  \begin{tikzpicture}
   \draw (-2.5,0) -- (2.5,0) ;
   \draw (0,-2.5) -- (0,2.5) ;
   \foreach \P in {(2.25,0),(1.5,0.9),(1.5,-0.9),(-0.2,1.4),(-0.2,-1.4),(-1.2,0.45),(-1.2,-0.45)}
      \node at \P [circle,fill,inner sep=1.5pt]{} ;
   \draw (1.2,-0.1) -- (1.2, 0.1) ;
   \draw node[below right] at (2.25,0) {$1$};
   \draw node[below left] at (1.2,0) {$\rho_0$};
   \draw node[below right] at (0,0) {$0$};
   \draw [dashed] (0,0) circle (2.25) ;
   \draw [dashed] (0,0) circle (1.2) ;
   \fill [black, opacity=0.5] (0,0) circle (1.2) ;
  \end{tikzpicture}
  }
  \caption{Spectrum of the operator $P$ acting on $\Bcal$. The set of Ruelle resonances is represented by the black dots, 
  and may be as small as $\{1\}$. The essential spectrum is inside the gray disk. The spectrum is 
  symmetric with respect to the real line because $P$ is real.}
\end{figure}

If the obstacles are close or the minimal curvature of the obstacles 
is small, then the estimate on the essential spectral radius becomes worse, which makes it harder to 
find resonances. The strategy we adopt let us work with those constants being fixed, avoiding this difficulty.

\subsection{Coverings of billiard tables}

A Sinai billiard table admits a $\Z^2$ covering, that is, a $\Z^2$-periodic billiard table $\widetilde{Q} \subset \R^2$ 
such that the natural projection of $\widetilde{Q}$ on $\Tbb^2$ is $Q$. In what follows, 
we denote with a tilde all objects related to this $\Z^2$-cover.

\smallskip

By choosing an origin and making $\Z^2$ act on $\widetilde{Q}$, the obstacles of $\widetilde{Q}$ 
can be indexed by $\Z^2$: we get $\R^2 \setminus \widetilde{Q} = \bigcup_{k=1}^d \bigcup_{p \in \Z^2} \mathring{\Gamma_{k,p}}$, 
where $\Gamma_{k,p} = \Gamma_{k,0}+p$, whence $\widetilde{M} = \bigcup_{p \in \Z^2} M \times \Z^2$. 

\smallskip

The Liouville measure $\mu$ also lifts to a $\widetilde{T}$-invariant measure $\widetilde{\mu}$, 
such that any restriction of $\widetilde{\mu}$ to a fundamental domain equals $\mu$.

\smallskip

The collision map $\widetilde{T}$ for $\widetilde{Q}$ is a $\Z^2$ extension of the collision 
map for $Q$. Given $(x,p) \in \widetilde{M} = M \times \Z^2$, we have $\widetilde{T} (x,p) =: (T(x),q)$, 
and $q-p$ depend only on $x$. Writing $q-p =: F(x)$, we get a function $F : M \to \Z^2$ such that:
\begin{equation*}
 \widetilde{T} (x,p) 
 = (T (x), p+F(x))
\end{equation*}
The value $F (x)$ stays the same as long as $(x,0)$ and $\widetilde{T} (x,0)$ 
belong to the same two obstacles. As a consequence, $F$ is constant on the domains 
of continuity of $T$. By the finite horizon condition, $T$ admits finitely many domains of continuity, 
so that $F$ is bounded. 

\smallskip

The billiard map is time-reversible. The involution $\iota (\ell, \theta) := (\ell, -\theta)$ 
on $M$ has the following properties:
\begin{align*}
 \iota \circ T 
 & = T^{-1} \circ \iota, \\
 F \circ \iota & = - F \circ T^{-1},
\end{align*}
and $\iota$ preserves $\mu$.  It follows that $\int_M F \dd \mu = 0$.

%
%
%

\smallskip

The same construction can be used 
on any covering of $Q$. In particular, given any rank $2$ lattice $\Lambda \subset \Z^2$, we get 
a billiard table $Q_\Lambda$, to which we associate a probability-preserving dynamical 
system $(M_\Lambda, \mu_\Lambda, T_\Lambda)$. Writing $G := \Z^2 / \Lambda$, we have:
\begin{align*}
 M_\Lambda & = M \times G, \\
 \mu_\Lambda & = \frac{1}{|G|} \sum_{g \in G} \mu \times \delta_g, \\
 T_\Lambda (x,g) & = (T(x), g+F(x) [\Lambda]).
\end{align*}

\begin{figure}[!h]
  \centering
  \scalebox{0.7}{
  \begin{tikzpicture}
   \draw [dashed, shift={(5,0)}] (-1,-1) -- (1,-1) -- (1,1) -- (-1,1) -- (-1,-1) ; 
   \draw [fill = black, opacity=0.5, shift={(5,0)}](0,0) circle (0.8) ; 
   \draw [fill = black, opacity=0.5, shift={(5,0)}] (-1,-1) -- (-0.5,-1) arc (0:90:0.5) -- cycle; 
   \draw [fill = black, opacity=0.5, shift={(5,0)}] (1,-1) -- (1,-0.5) arc (90:180:0.5) -- cycle;
   \draw [fill = black, opacity=0.5, shift={(5,0)}] (1,1) -- (0.5,1) arc (180:270:0.5) -- cycle;
   \draw [fill = black, opacity=0.5, shift={(5,0)}] (-1,1) -- (-1,0.5) arc (270:360:0.5) -- cycle;
   \draw [dashed] (-2,-3) -- (2,-3) -- (2,3) -- (-2,3) -- (-2,-3) ; 
   \foreach \P in {(-1,-2),(1,-2),(-1,0),(1,0),(-1,2),(1,2)} 
      \draw [fill = black, opacity=0.5] \P circle (0.8) ;
   \foreach \P in {(0,-1),(0,1)}
      \draw [fill = black, opacity=0.5] \P circle (0.5) ;
   \draw [fill = black, opacity=0.5] (-2,-3) -- (-1.5,-3) arc (0:90:0.5) -- cycle; 
   \draw [fill = black, opacity=0.5] (2,-3) -- (2,-2.5) arc (90:180:0.5) -- cycle;
   \draw [fill = black, opacity=0.5] (2,3) -- (1.5,3) arc (180:270:0.5) -- cycle;
   \draw [fill = black, opacity=0.5] (-2,3) -- (-2,2.5) arc (270:360:0.5) -- cycle;
   \draw [fill = black, opacity=0.5] (0,-3) -- (0.5,-3) arc (0:180:0.5) -- cycle; 
   \draw [fill = black, opacity=0.5] (0,3) -- (-0.5,3) arc (180:360:0.5) -- cycle;
   \draw [fill = black, opacity=0.5] (-2,-1) -- (-2,-1.5) arc (-90:90:0.5) -- cycle;
   \draw [fill = black, opacity=0.5] (-2,1) -- (-2,0.5) arc (-90:90:0.5) -- cycle;
   \draw [fill = black, opacity=0.5] (2,-1) -- (2,-0.5) arc (90:270:0.5) -- cycle;
   \draw [fill = black, opacity=0.5] (2,1) -- (2,1.5) arc (90:270:0.5) -- cycle;
   \foreach \n in {-2,0,2} 
      \foreach \m in {-3,-1,1,3}
         \draw [fill = black, opacity=0.5, shift={(-8,0)}] (\n,\m) circle (0.8) ;
   \foreach \n in {-3,-1,1,3} 
      \foreach \m in {-4,-2,0,2,4}
         \draw [fill = black, opacity=0.5, shift={(-8,0)}] (\n,\m) circle (0.5) ;
   \draw [->, line width=1mm] (-3.5,0) -- (-2.5,0) ;
   \draw [->, line width=1mm] (2.5,0) -- (3.5,0) ;
  \end{tikzpicture}
  }
  \caption{The $\Z^2$-covering $\widetilde{Q}$ of $Q$, and an intermediate covering $Q_\Lambda$ with $\Lambda=2\Z \oplus 3\Z$.}
\end{figure}
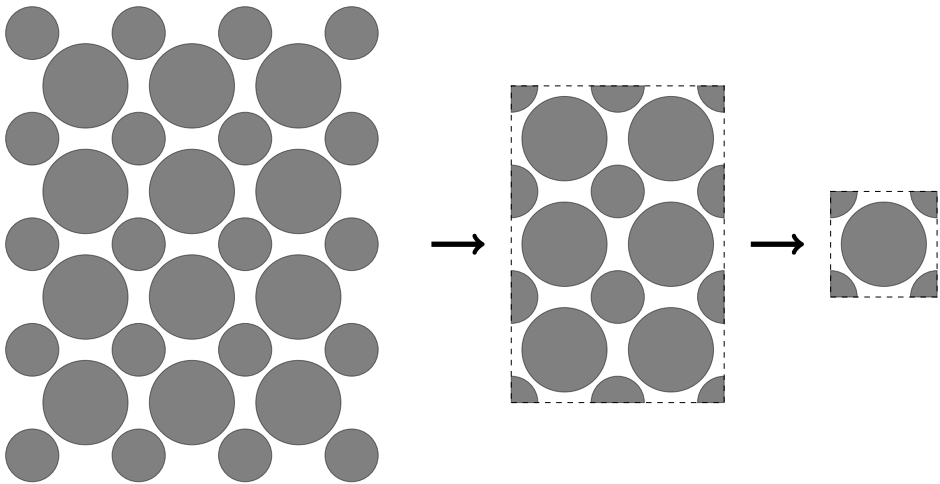

We denote by $P_\Lambda$ the transfer operator associated with $(M_\Lambda, \mu_\Lambda, T_\Lambda)$, 
and $\Bcal_\Lambda$ the Banach space as constructed in~\cite{DemersZhang:2012} for the system $(M_\Lambda, \mu_\Lambda, T_\Lambda)$. 
By construction, for any $h \in \Bcal_\Lambda$,
\begin{equation*}
 \norm{h}{\Bcal_\Lambda} 
 = \max_{g \in G} \norm{h \mathbf{1}_{M \times \{g\}}}{\Bcal}.
\end{equation*}

\subsection{Results}

We shall prove that, when $\Lambda$ is large enough, the  Sinai billiard on 
the table $Q_\Lambda$ has non-trivial Ruelle resonances. More precisely,

\begin{theorem}
\label{thm:ExistenceResonances}
  
 There exists $\delta >0$ such that 
 $\Sp (P_\Lambda \acts \Bcal_\Lambda) \subset \overline{B}(0,1-\delta) \cup [1-\delta, 1]$ 
 for all lattices $\Lambda$.
 
 \smallskip
 
 In addition, there exist positive constants $c<C$ such that:
 \begin{equation*}
  c |G| 
  < |\{\text{Ruelle resonances in } [1-\delta, 1], \text{ with multiplicities}\}|
  < C|G|
 \end{equation*}
 In particular,  whenever $\Lambda$ is sparse enough, $P_\Lambda$ admits non-trivial 
 Ruelle resonances. 
\end{theorem}

\begin{figure}[!h]
  \centering
  \scalebox{1}{
  \begin{tikzpicture}
   \draw (-2.5,0) -- (2.5,0) ;
   \draw (0,-2.5) -- (0,2.5) ;
   \foreach \P in {(2.25,0),(2.15,0),(2.05,0),(1.95,0)}
      \node at \P [circle,fill,inner sep=1.5pt]{} ;
   \foreach \P in {(1.5,0.9),(1.4,0.7),(1.2,0.7), (1,0.9),(1.3,1.05),(1.15,1.05)}
      \node at \P [circle,fill,inner sep=1.5pt]{} ;
   \foreach \P in {(1.5,-0.9),(1.4,-0.7),(1.2,-0.7), (1,-0.9),(1.3,-1.05),(1.15,-1.05)}
      \node at \P [circle,fill,inner sep=1.5pt]{} ;
   \foreach \P in {(-0.2,1.4),(0,1.3),(-0.4,1.2)}
      \node at \P [circle,fill,inner sep=1.5pt]{} ;
   \foreach \P in {(-0.2,-1.4),(0,-1.3),(-0.4,-1.2)}
      \node at \P [circle,fill,inner sep=1.5pt]{} ;
   \foreach \P in {(-1.2,0.45),(-1.2,-0.45)}
      \node at \P [circle,fill,inner sep=1.5pt]{} ;
   \draw (1.2,-0.1) -- (1.2, 0.1) ;
   \draw node[below right] at (2.25,0) {$1$};
   \draw node[below right] at (1.2,0) {$1-\delta$};
   \draw node[below left] at (1.2,0) {$\rho_0$};
   \draw node[below right] at (0,0) {$0$};
   \draw [dashed] (0,0) circle (2.25) ;
   \draw [dashed] (0,0) circle (1.8) ;
   \draw [dashed] (0,0) circle (1.2) ;
   \fill [black, opacity=0.5] (0,0) circle (1.2) ;
  \end{tikzpicture}
  }
  \caption{Spectrum of the operator $P_\Lambda$ acting on $\Bcal_\Lambda$ for a sparse enough lattice $\Lambda$. 
  The spectrum is still symmetric with respect to the real line. The resonances on the segment $[1- \delta,1]$ are guaranteed to exist; 
  the others may or may not exist.}
\end{figure}
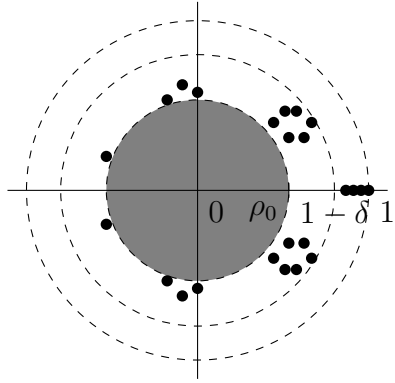


Now, let us focus on the distribution of these resonances. 
Let $\rho_0$ be the upper bound on $\rho_{\ess} (P \acts \Bcal)$ given by~\cite[Proposition~2.3]{DemersZhang:2012}. 
For $\Lambda < \Z^2$ a rank $2$ lattice, define the \textit{spectral measure} of $P_\Lambda$ as:
\begin{equation*}
 \nu_\Lambda 
 := \frac{1}{|G|} \sum_{\substack{\lambda \text{ resonance of } P_\Lambda \\ |\lambda| > 1-\rho_0}} \delta_\lambda,
\end{equation*}
where the sum is taken with multiplicity. Then $\rho_{\ess} (P_\Lambda \acts \Bcal_\Lambda) \leq \rho_0$, 
so that $\nu_\Lambda$ is a Radon measure on $\overline{B}(0,\rho_0)^c$. Our next proposition, 
which is a variant of~\cite[Theorem~1.3]{JakobsonNaudSoares:2019}, states that, for 
any sequence $(\Lambda_N)$ of lattices, the sequence $(\nu_{\Lambda_N})$ of spectral measures 
admits a converging subsequence.

\begin{proposition}
\label{prop:CompactenessSpectralMeasures}
 
 For any sequence of rank $2$ lattices $\Lambda_N < \Z^2$, there exists 
 a subsequence $(\Lambda_{N_k})_{k \geq 0}$ and a Radon measure $\nu$ such that 
 $\nu_{\Lambda_{N_k}} \to \nu$ for the vague topology, 
 i.e.\
 \begin{equation*}
  \lim_{k \to + \infty} \int_{\overline{B}(0,\rho_0)^c} f \dd \nu_{\Lambda_{N_k}} 
  = \int_{\overline{B}(0,\rho_0)^c} f \dd \nu
 \end{equation*}
 for all $f \in \Ccal_c (\overline{B}(0,\rho_0)^c, \C)$.
\end{proposition}

For specific choices of a sequence of lattices $(\Lambda_N)$, we can express explicitly the 
limit of $(\nu_{\Lambda_N})$ near $1$, similarly to what was done in~\cite[Section~3.2]{JakobsonNaudSoares:2019}. 
Let $\Lambda_N^{(1)} := N \Z \times \Z$ and $\Lambda_N^{(2)} := (N \Z)^2$. 
We write $\nu_N^{(1)} := \nu_{\Lambda_N^{(1)}}$ and likewise $\nu_N^{(2)} := \nu_{\Lambda_N^{(2)}}$.
For the sequence of lattices $(\Lambda_N^{(1)})$, we get the following statement:

\begin{proposition}
\label{prop:DistributionResonancesDim1}
 
 Let $\delta$ be as in Theorem~\ref{thm:ExistenceResonances}.
 There exists $\delta_0 \in (0, \delta]$ and a finite measure 
 $\nu_{|[1-\delta_0,1]}^{(1)}$ on $[1-\delta_0, 1]$ such that:
 \begin{equation*}
  \lim_{N \to + \infty} \nu_{N|[1-\delta_0,1]}^{(1)} 
  = \nu_{|[1-\delta_0,1]}^{(1)},
 \end{equation*}
 where the convergence is in $\Ccal([1-\delta_0,1], \C)^*$. Moreover,
 \begin{equation}
 \label{eq:Distribution1}
  \frac{\de \nu_{|[1-\delta_0,1]}^{(1)}}{\de x} 
  \sim_{x \to 1^-} \frac{1}{\pi \sqrt{2 \Sigma_{11}}} \cdot \frac{1}{\sqrt{1-x}}.
 \end{equation}
\end{proposition}

In Equation~\eqref{eq:Distribution1}, the constant $\Sigma_{11}$ is the 
upper-left coefficient of the covariance matrix $\Sigma$ associated with the diffusion 
of the billiard on the table $\widetilde{Q}$ (see Equation~\eqref{eq:NagaevGuivarch}).
Replacing $\Lambda_N^{(1)}$ by $\Z \times N \Z$ only changes the constant 
in Equation~\eqref{eq:Distribution1}, where $\Sigma_{11}$ becomes $\Sigma_{22}$. 

\smallskip

We get an analogous statement for the sequence $(\Lambda_N^{(2)})$:

\begin{proposition}
\label{prop:DistributionResonancesDim2}
 
 Let $\delta$ be as in Theorem~\ref{thm:ExistenceResonances}.
 There exists $\delta_0 \in (0, \delta]$ and a finite measure 
 $\nu_{|[1-\delta_0,1]}^{(2)}$ on $[1-\delta_0, 1]$ such that:
 \begin{equation*}
  \lim_{N \to + \infty} \nu_{N|[1-\delta_0,1]}^{(2)} 
  = \nu_{|[1-\delta_0,1]}^{(2)},
 \end{equation*}
 where the convergence is in $\Ccal([1-\delta_0,1], \C)^*$. Moreover,
 \begin{equation}
 \label{eq:Distribution2}
  \lim_{x \to 1^-} \frac{\de \nu_{|[1-\delta_0,1]}^{(2)}}{\de x} 
  = \frac{1}{2 \pi \sqrt{\det (\Sigma)}}.
 \end{equation}
\end{proposition}

\section{Existence of resonances}
\label{sec:ExistenceResonance}

In this section, we prove a weaker version of Theorem~\ref{thm:ExistenceResonances}:

\begin{proposition}
 \label{prop:ExistenceResonancesFaible}
 
 Let $(M,\mu,T)$ be a finite horizon Sinai billiard.
 Let $U$ be a neighborhood of $1$ in $\C$. There exists a constant 
 $c(U)>0$ such that, for any rank $2$ lattice $\Lambda<\Z^2$, 
 the spectrum of $P_\Lambda$ acting on $\Bcal_\Lambda$ admits at least 
 $c(U) |G|$ Ruelle resonances (with multiplicities) 
 in $U$.
 
 \smallskip
 
 In particular, if $\Lambda$ is sparse enough, 
 then $P_\Lambda$ admits non-trivial Ruelle resonances in $U$.
\end{proposition}

Theorem~\ref{thm:ExistenceResonances} shall follow from Proposition~\ref{prop:ExistenceResonancesFaible} 
and some additional results on the localization of Ruelle resonances proved in Section~\ref{sec:DistributionResonances}.

\subsection{Spectral decomposition of $\Bcal_\Lambda$}
\label{subsec:ExistenceResonance}

The map $M_\Lambda \to M$ is a Galois covering, the deck transformations being 
translations $\tau_g : M_\Lambda \to M_\Lambda$ given by $\tau_g (x, g') = (x, g'+g)$, 
for all $g \in G$. All these translations are measure-preserving and commute with $T_\Lambda$:
\begin{equation*}
 \tau_g \circ T_\Lambda (x, g') 
 = (T(x), g'+F(x)+g [\Lambda])
 =  T_\Lambda \circ \tau_g (x, g').
\end{equation*}
As a consequence, $G$ acts on $\Bcal_\Lambda$ by pre-composition:
\begin{equation*}
 (\tau_g h_\Lambda) (\varphi)
 := h_\Lambda (\varphi \circ \tau_{-g})
\end{equation*}
for all $h_\Lambda \in \Bcal_\Lambda$ and $\varphi \in \Ccal^1 (M_\Lambda, \C)$, 
and this action commutes with $P_\Lambda$.

\smallskip

Let $g \in G$. Since $\tau_g$ and $P_\Lambda$ commute, the eigenspaces of $\tau_g$ 
are $P_\Lambda$-invariant. As this holds for all $g \in\tau_g$,  
the intersections of the eigenspaces for $(\tau_g)_{g \in G}$ are $P_\Lambda$-invariant. 
But these eigenspaces are given by the characters of $G$:
\begin{equation*}
 \Bcal_{\Lambda, \chi} 
 := \{h \otimes \chi: \ h \in \Bcal_\Lambda, \ \chi \in \widehat{G}\},
\end{equation*}
where $(h \otimes \chi) (\varphi_\Lambda) = \sum_{g \in G} \chi (g) h(\varphi_\Lambda (\cdot, g))$ for all $\varphi_\Lambda \in \Ccal^1 (M_\Lambda, \C)$. 
Note that we can define maps $\Pi_{\Lambda, \chi} : \Bcal_\Lambda \to \Bcal$ by:
\begin{equation*}
 \Pi_{\Lambda, \chi} (h_\Lambda) (\varphi) 
 := |G|^{-1} h_\Lambda (\varphi \otimes \overline{\chi}) \quad \forall \varphi \in \Ccal^1 (M, \C),
\end{equation*}
and the space $\Bcal_{\Lambda, \chi}$ can be written as:
\begin{equation*}
 \Bcal_{\Lambda, \chi}
 = \bigcap_{\substack{\chi \in \widehat{G} \\ \chi' \neq \chi}} \Ker (\Pi_{\Lambda, \chi'}).
\end{equation*}
In particular, the spaces $\Bcal_{\Lambda, \chi}$ are closed, and there is a $P_\Lambda$-invariant splitting:
\begin{equation*}
 \Bcal_\Lambda 
 = \bigoplus_{\chi \in \widehat{G}} \Bcal_{\Lambda, \chi}.
\end{equation*}

\subsection{Spectrum of $P_\Lambda$}

Each subspace $\Bcal_{\Lambda, \chi}$ is isomorphic to $\Bcal$, for instance \textit{via} $\Pi_{\Lambda, \chi}$ 
and its right inverse $h \mapsto h \otimes \chi$. 
The action of $P_\Lambda$ on $\Bcal_{\Lambda, \chi}$ is thus conjugated with the action 
of some operator $P_{\Lambda, \chi}$ on $\Bcal$, which can be made explicit. For all $h \in \Bcal$ and 
$\varphi \in \Ccal^1 (M, \C)$,
\begin{align*}
 P_{\Lambda, \chi} (h) (\varphi) 
 & = [\Pi_{\lambda, \chi} P_\Lambda (h \otimes \chi)] (\varphi) \\
 & = \frac{1}{|G|} [P_\Lambda(h \otimes \chi)] (\varphi \otimes \overline{\chi}) \\
 & = \frac{1}{|G|} (h \otimes \chi) (\varphi \otimes \overline{\chi} \circ T_\Lambda) \\
 & = \frac{1}{|G|} (h \otimes \chi) [(\overline{\chi} (F) \cdot \varphi \circ T) \otimes \overline{\chi}] \\
 & = h (\chi (-F) \cdot \varphi \circ T) \\
 & = [P ( \chi (-F) h )] (\varphi).
\end{align*}
Hence, $P_{\Lambda, \chi} (h) = P ( \chi (-F) h )$. More intuitively, for $h \in \Ccal^0 (M, \C)$, 
we have $P_{\Lambda, \chi} (h \otimes \chi) = (h \otimes \chi) \circ T_\Lambda^{-1}$, with:
\begin{equation*}
 T_\Lambda^{-1} (x,p) 
 = (T^{-1} (x), p-F \circ T^{-1} (x)),
\end{equation*}
from which the same result follows. As a consequence,
\begin{equation}
 \label{eq:EgaliteSpectres}
 \Sp (P_\Lambda \acts \Bcal_\Lambda) 
 = \bigcup_{\chi \in \widehat{G}} \Sp (P_{\Lambda, \chi} \acts \Bcal), 
\end{equation}
where the union is taken with multiplicities.

\smallskip

Note that the estimate given in~\cite[Proposition~2.3]{DemersZhang:2012} on the essential spectral radius 
of $P$ also holds for $P_\Lambda$, for all $\Lambda$. There are two ways to prove this theorem:
\begin{itemize}
 \item That bound only depends on the minimal curvature of the obstacles, 
 the minimal free path length of the bouncing particle, and on the choice of $\varepsilon_0$. 
 These quantities are the same for the billiard table $M$ and for all its covers $M_\Lambda$.
 \item The estimates on the essential spectral radius of~\cite[Proposition~2.3]{DemersZhang:2012} 
 generalize to the weighted operators $P_{\Lambda, \chi}$, with no dependence on the character 
 $\chi$. Equation~\eqref{eq:EgaliteSpectres} yields the claim.
\end{itemize}

\subsection{Perturbations of transfer operators}

The function $F : M \to \Z^2$ satisfies the assumptions of~\cite[Lemma~5.4]{DemersZhang:2014}. 
Hence, the family of twisted transfer operators:
\begin{equation*}
 P_w (h) 
 := P (e^{i \langle w, F \rangle} h),
\end{equation*}
acting on $\Bcal$, depends analytically on $w \in 2 \pi \Tbb^2$.

\smallskip

As $1$ is an isolated eigenvalue of $P$ with eigenfunction $\mathbf{1}$, 
there exists a neighbourhood $U$ of $0$ in $2 \pi \Tbb^2$ and analytic functions 
$w \mapsto \lambda_w \in \C$, and $w \mapsto g_w \in \Bcal$ defined on $U$ such 
that $g_w (\mathbf{1}) = 1$ and $P_w (h) = \lambda_w g_w$.

\smallskip

By a classical computation, appearing for instance in the Nagaev-Guivarc'h proof of the Central Limit Theorem 
for Markov chains~\cite{Nagaev:1957, Nagaev:1961, GuivarchHardy:1988}, 
\begin{equation}
 \label{eq:NagaevGuivarch}
 \lambda_w 
 = 1 - \frac{\Sigma (w,w)}{2} + O (|w|^3)
\end{equation}
near $0$, where $\Sigma$ is a bilinear form. As $\Sigma$ is the matrix of covariance for the 
central limit theorem for $(\sum_{k=0}^{n-1} F \circ T^k)$ and such a 
central limit theorem for a finite horizon Sinai billiard has a non-degenerate limit law, 
$F$ is not a coboundary and $\Sigma$ is positive definite. In particular, up to taking a smaller neighborhood $U$, 
we shall assume that $|\lambda_w|<1$ for $w \neq 0$.

\subsection{Ruelle resonances for Sinai billiards}

We now prove Proposition~\ref{prop:ExistenceResonancesFaible}. 

\begin{proof}[Proposition~\ref{prop:ExistenceResonancesFaible}]
 
Let $U$ be a neighborhood of $1$ in $\C$. Let $V$ be a neighborhood 
of $0$ in $2 \pi \Tbb^2$ on which the main eigenvalue $\lambda_w$ of $P_w$ is well-defined 
and belongs to $U$. Let $W$ be a neighborhood of $0$ such that 
$W-W \subset V$. Let $\Lambda < \Z^2$. Note that $\widehat{G} = \widehat{(\Z^2 / \Lambda)} <  \widehat{\Z^2} = 2 \pi \Tbb^2$. 
Then:
\begin{equation}
 \label{eq:BorneNombreResonances0}
 |G| \Vol (W) 
 = \int_{2 \pi \Tbb^2} \sum_{\chi \in \widehat{G}} \mathbf{1}_{\chi+W} \dd \Vol 
 \leq 4 \pi^2 \max_{2 \pi \Tbb^2} \sum_{\chi \in \widehat{G}} \mathbf{1}_{\chi+W}.
\end{equation}
In addition, for all $\chi' \in 2 \pi \Tbb^2$, 
\begin{equation*}
 \sum_{\chi \in \widehat{G}} \mathbf{1}_{\chi+W} (\chi') 
 = |\widehat{G} \cap (\chi'-W)|.
\end{equation*}
Assume that $\widehat{G} \cap (\chi'-W)$ is non-empty, and let $\chi_0$ be one of its elements. 
The function $\chi \mapsto |\widehat{G} \cap (\chi-W)|$ is $\widehat{G}$-invariant, 
so $ |\widehat{G} \cap (\chi'-W)| =  |\widehat{G} \cap (\chi'-\chi_0-W)|$. But $\chi_0 \in \chi'-W$, 
so $\chi'-\chi_0 \in W$ and $\chi'-\chi_0-W \subset W-W \subset V$. Hence, $|\widehat{G} \cap (\chi'-W)| \leq |\widehat{G} \cap V|$. 
This inequality is also true if $\widehat{G} \cap (\chi'-W)$ is empty, so Equation~\eqref{eq:BorneNombreResonances0} 
implies:
\begin{equation}
 \label{eq:BorneNombreResonances}
 |G| \Vol (W) 
 \leq 4 \pi^2 |\widehat{G} \cap V|.
\end{equation}
Finally, given any $\chi = e^{i \langle w, \cdot \rangle}\in \widehat{G} \cap V$, 
the operator $P_w = P_{\Lambda, \overline{\chi}}$ admits $\lambda_w$ 
as a Ruelle resonance. By construction, $\lambda_w \in U$, and $\lambda_w$ 
is also a resonance of $P_\Lambda$ by Equation~\eqref{eq:EgaliteSpectres}. 
Hence, $P_\Lambda$ admits at least $|\widehat{G} \cap V| \geq \frac{\Vol (W)}{4\pi^2}|G|$ 
Ruelle resonances in $U$ (with multiplicities).
\end{proof}

\section{Distribution of resonances}
\label{sec:DistributionResonances}

We now focus on the properties of the non-trivial resonances constructed in Section~\ref{sec:ExistenceResonance}. 
We shall finish the proof of Theorem~\ref{thm:ExistenceResonances} in Subsections~\ref{subsec:BillardAperiodique} 
and~\ref{subsec:ResonancesReelles}, where we describe more precisely the main eigenvalue $\lambda_w$ 
of the twisted transfer operator $P_w$. We shall prove Propositions~\ref{prop:CompactenessSpectralMeasures}, 
\ref{prop:DistributionResonancesDim1} and~\ref{prop:DistributionResonancesDim2} in Subsection~\ref{subsec:DistributionResonances}.

\subsection{Aperiodicity of the Sinai billiard}
\label{subsec:BillardAperiodique}

As a step-stone to Theorem~\ref{thm:ExistenceResonances}, we shall prove that Sinai billiards are aperiodic. 
While this is already known~\cite{SzaszVarju:2004}, the following argument is reasonably short. 
We write characters in exponential form: $\chi = e^{i \langle w, \cdot \rangle}$. 
Let $H := \{(\rho, w) \in \Sbb_1 \times 2\pi \Tbb^2 | \ \rho \in \Sp (P_w \acts \Bcal) \}$. We claim:

\begin{lemma}
\label{lem:SpectrePeripherique}
 
 The set $H$ is a subgroup of $\Sbb_1 \times 2\pi \Tbb^2$. 
\end{lemma}

\begin{proof}
 
 Note that, by considering $H \cap (\Sbb_1 \times \{0\})$, this lemma implies 
 that the peripheral spectrum of $P$ is a subgroup of $\Sbb_1$, which is well-known~\cite[Lemma~5.2]{DemersZhang:2012}. 
 The proof of Lemma~\ref{lem:SpectrePeripherique} mimics the proof of the later fact.
 
 \smallskip
 
 Up to straightforward modifications, the proof of~\cite[Lemma~5.1]{DemersZhang:2012} 
 can be generalized to prove that, for any $(\rho, w) \in H$, the corresponding Jordan block 
 is trivial and any associated eigendistribution belongs to $\Lbb^\infty (M, \mu) \cdot \de \mu$. 
 In addition, whenever $k$ is an eigendistribution associated with $(\rho, w) \in H$, 
 \begin{equation*}
  P_w (k \dd \mu) 
  = e^{i \langle w, F \rangle \circ T^{-1}} k \circ T^{-1} \dd \mu
  = \rho k \dd \mu,
 \end{equation*}
 and thus 
 \begin{equation}
  \label{eqref:EquationFonctionnelleRhoH}
  \rho k 
  = e^{i \langle w, F \rangle \circ T^{-1}} k \circ T^{-1}.
 \end{equation}

 Taking absolute values in Equation~\eqref{eqref:EquationFonctionnelleRhoH}, 
 we get $|k| \circ T^{-1} = |k|$; as the Sinai billiard is ergodic, 
 $|k|$ is constant, and $k$ does not vanish.
 
 \smallskip
 
 Taking the  complex conjugate in Equation~\eqref{eqref:EquationFonctionnelleRhoH}, 
 we get $\overline{\rho} \overline{k} = e^{i \langle -w, F \rangle \circ T^{-1}} \overline{k} \circ T^{-1}$, 
 so $\overline{k} \dd \mu$ is an eigendistribution for $P_{-w}$ for the eigenvalue $\overline{\rho}$. In addition, 
 $\overline{k} \dd \mu \in \Bcal$. Hence, $(\overline{\rho}, -w) \in H$.
 
 \smallskip

 Let $k_1 \dd \mu$, $k_2 \dd \mu$ be two eigendistributions corresponding to $(\rho_1,w_1)$, 
 $(\rho_2, w_2) \in H$ respectively. Then, again using Equation~\eqref{eqref:EquationFonctionnelleRhoH},
 \begin{equation*}
  e^{i \langle w_1+w_2, F \rangle \circ T^{-1}} (k_1 k_2) \circ T^{-1} 
  = e^{i \langle w_1, F \rangle \circ T^{-1}} k_1 \circ T^{-1} \cdot e^{i \langle w_2, F \rangle \circ T^{-1}} k_2 \circ T^{-1} 
  = \rho_1 k_1 \rho_2 k_2 
  = (\rho_1 \rho_2) k_1 k_2.
 \end{equation*}
 As neither $k_1$ not $k_2$ vanish, their product $k_1 k_2$ does not vanish either, 
 so $k_1 k_2$ is an eigendistribution for the eigenvalue $\rho_1 \rho_2$ of $P_{w_1+w_2}$. 
 In addition,  by the argument of~\cite[Lemma~5.5]{DemersLiverani:2008}, $k_1 k_2 \dd \mu \in \Bcal$. 
 Hence, $(\rho_1 \rho_2, w_1+w_2) \in H$. 
\end{proof}

All the arguments above are standard (they only use properties of the action of $P_w$ on $\Bcal$), 
and apply to a much wider class of dynamical systems. For Sinai billiards, Lemma~\ref{lem:SpectrePeripherique} 
can be strengthened:

\begin{lemma}
\label{lem:Aperiodicite}
 
 For a finite horizon Sinai billiard, $H = \{(1,0)\}$. 
\end{lemma}

\begin{proof}
 
 By the discussion in Subsection~\ref{subsec:ExistenceResonance}, 
 there exists a neighborhood $V$ of $0$ in $2 \pi \Tbb^2$ and a neighborhood 
 $U$ of $1$ in $\C$ such that, for all $w \in V$, we have $\Sp (P_w \acts \Bcal) \cap U = \{\lambda_w\}$. 
 Since $\lambda_w = 1 - \frac{\Sigma (w,w)}{2} + O (|w|^3)$, if $w \in V \setminus \{0\}$, then $P_w$ 
 has no eigenvalue of modulus $1$ in $U$. Hence, $(1,0)$ is isolated in $H$. The group $H$ is discrete, 
 and thus finite.
 
 \smallskip
 
 Assume that $H$ is not trivial, and let $(\rho, w) \in H \setminus \{(1,0)\}$. Since $H$ is finite, 
 $(\rho, w)$ has finite order. Hence, there exists a rank $2$ lattice $\Lambda < \Z^2$ such that $w \in \widehat{G}$. 
 Then $\rho \in \Sp (P_{\Lambda, \chi} \acts \Bcal)$ for $\chi = e^{- i \langle w, \cdot \rangle} \in \widehat{G}$, so 
 $\rho \in \Sp (P_\Lambda \acts \Bcal_\Lambda)$. Then the transfer operator $P_\Lambda$ for the Sinai billiard $Q_\Lambda$ 
 has non-trivial peripheral spectrum. If $\rho = 1$, then $1 \in \Sp (P_\Lambda \acts \Bcal_\Lambda)$ has multiplicity at least $2$, 
 which contradicts the ergodicity of Sinai billiards. If $\rho \neq 1$, then $\rho$ is a non-trivial root 
 of the unit, which contradicts the fact that Sinai billiards are mixing.
\end{proof}

The Ruelle spectrum of $w \mapsto P_w$ depends continuously on $w$. By Lemma~\ref{lem:Aperiodicite}, 
there exists $\delta >0$ and a neighborhood $V$ of $0$ in $2 \pi \Tbb^2$ 
such that $P_w$ has a resonance of modulus larger than $1-\delta$ if and only if 
$w \in V$, and under this condition the resonance is $\lambda_w$. 

\smallskip

By Equation~\eqref{eq:EgaliteSpectres}, the spectrum of $P_\Lambda$ acting on $\Bcal_\Lambda$ 
is the union of the spectra of the operators $P_w$ for $w \in \widehat{G}$. Hence, 
for this value of $\delta >0$ and all rank $2$ lattices $\Lambda < \Z^2$:
\begin{equation}
 \label{eq:StructureSpectre0}
 \Sp (P_\Lambda \acts \Bcal_\Lambda) 
 \subset \overline{B}(0,1-\delta) \cup \{\lambda_w : \ w \in V\}.
\end{equation}
A further consequence is that $P_\Lambda$ admits at most $|G|$ resonances (with multiplicities) 
in $\overline{B}(0,1-\delta)^c$. With Proposition~\ref{prop:ExistenceResonancesFaible}, this implies that the number 
of Ruelle resonances in the annulus $\{1-\delta < |z| \leq 1\}$ is a $\Theta (|G|)$.

\subsection{Realness of the resonances}
\label{subsec:ResonancesReelles}

Let $\delta>0$ be small enough that Equation~\eqref{eq:StructureSpectre0} holds. 
Let $V := \{w \in 2 \pi \Tbb^2 : \ |\lambda_w|>1-\delta\}$. Up to taking a smaller 
value of $\delta$, we may assume that $\lambda_\cdot$ is continuous on $V$. 
Let $w \in V$, and $h_w \in \Bcal$ an eigendistribution for $P_w$.
Since the operator $P$ is real,
\begin{equation*}
 P_{-w} (\overline{h}_w) 
 = P (e^{-i \langle w, F \rangle} \overline{h}_w) 
 = \overline{P (e^{i \langle w, F \rangle} h_w)} 
 = \overline{\lambda}_w \overline{h}_w,
\end{equation*}
so $\overline{\lambda}_w \in \Sp (P_{-w} \acts \Bcal)$. But $|\overline{\lambda}_w|>1-\delta$ 
and the only eigenvalue of $P_{-w}$ of modulus larger than $1-\delta$ is $\lambda_{-w}$. 
Hence, $\lambda_{-w} = \overline{\lambda}_w$ for all $w \in V$.

\smallskip

We recall that the billiard map is time-reversible. The involution  $\iota (\ell, \theta) = (\ell, -\theta)$ 
satifies:
\begin{align*}
 \iota \circ T 
 & = T^{-1} \circ \iota, \\
 F \circ \iota & = - F \circ T^{-1},
\end{align*}

\smallskip

Let $\Bcal^*$ be the dual of $\Bcal$. Informally, the space $\Bcal^*$ contains distributions which 
are regular in the direction of the stable cones of $T$, 
and irregular in the unstable cones. Then $\Sp (P_w^* \acts \Bcal^*) = \Sp (P_w \acts \Bcal)$.
In addition, for all $\varphi$, $\psi \in \Ccal^1 (M, \C)$:
\begin{equation}
 \label{eq:ExpressionPwDual}
 \int_M P_w^* (\varphi) \cdot \psi \dd \mu 
 = \int_M \varphi \cdot P_w (\psi) \dd \mu 
 = \int_M e^{i \langle w, F \rangle} \varphi \circ T \cdot  \psi \dd \mu.
\end{equation}

\smallskip

Let $\widetilde{\Bcal}$ be the image of $\Bcal$ under precomposition by the involution $\iota$. 
Again, informally, $\widetilde{\Bcal}$ contains distributions which 
are regular in the direction of the stable cones of $T$, and irregular in the unstable cones. 
Let us define $\widetilde{P}_w (\varphi) := (P_w (\varphi \circ \iota)) \circ \iota$, and extend this operator 
by continuity to $\widetilde{\Bcal}$. Then $\Sp (\widetilde{P}_w \acts \widetilde{\Bcal}) = \Sp (P_w \acts \Bcal)$. 
In addition, for all $\varphi$, $\psi \in \Ccal^1 (M, \C)$:
\begin{align}
 \int_M \widetilde{P}_w (\varphi) \cdot \psi \dd \mu 
 & = \int_M P_w (\varphi \circ \iota) \cdot \psi \circ \iota \dd \mu \nonumber \\
 & = \int_M e^{i \langle w, F \rangle} \varphi \circ \iota \cdot \psi \circ \iota \circ T \dd \mu \nonumber \\
 & = \int_M e^{i \langle w, - F \circ T^{-1} \circ \iota \rangle} \varphi \circ \iota \cdot \psi \circ T^{-1} \circ \iota \dd \mu \nonumber \\
 & = \int_M e^{-i \langle w, F \circ T^{-1} \rangle} \varphi \cdot \psi \circ T^{-1} \dd \mu \nonumber \\
 & = \int_M e^{-i \langle w, F \rangle} \varphi \circ T \cdot \psi \dd \mu. \label{eq:ExpressionPwTilde}
\end{align}

To sum up:
\begin{align*}
 \Sp (P_{-w}^* \acts \Bcal^*) & = \Sp (P_{-w} \acts \Bcal), \\
 \Sp (\widetilde{P}_w \acts \widetilde{\Bcal}) & = \Sp (P_w \acts \Bcal),
\end{align*}
and the operators $P_{-w}^*$ and $\widetilde{P}_w$ coincide on $\Ccal^1$ functions 
by Equations~\eqref{eq:ExpressionPwDual} and~\eqref{eq:ExpressionPwTilde}. 

\smallskip

We would like to show that the spectra of $P_{-w}^*$ and $\widetilde{P}_w$ coincide, at least outside 
of $\overline{B} (0,\rho_0)$. Unfortunately, a result such as~\cite[Lemma~A.1]{BaladiTsujii:2008} is 
not directly applicable, because we don't know whether $\Ccal^1 (M, \C)$ is dense in $\Bcal^*$. We 
use instead an ad hoc argument, and show the weaker result $\lambda_w = \lambda_{-w}$.

\smallskip

The space $\widetilde{\Bcal}$ is defined as the completion of $\Ccal^1 (M,\C)$ for the norm 
$\norm{\cdot}{\widetilde{\Bcal}}$. Hence, $\Ccal^1 (M,\C)$ is dense for the strong topology 
on $\widetilde{\Bcal}$. In addition, by~\cite[Lemma~3.9]{DemersZhang:2012}, $\Ccal^1 (M,\C)$ 
maps continuously into $\Bcal^*$, and this map is injective (since there is an injective 
embedding $\Ccal^\gamma \hookrightarrow \Bcal$).

\smallskip

Let $\Pi_{-w}$ be the (rank $1$) spectral projection of $P_{-w}$ onto the eigenspace 
corresponding to the eigenvalue $\lambda_\varepsilon$, and $Q_{-w} := P_{-w}-\lambda_w \Pi_{-w}$. 
The spectral radius of $Q_{-w}$ is no larger than $1-\delta$, so let $\delta'<\delta$ with $|\lambda_{-w}| >1-\delta'$. 
Then, for all $\varphi$, $\psi \in \Ccal^1 (M,\C)$,
\begin{equation*}
 \int_M \varphi \cdot P_{-w}^n (\psi) \dd \mu 
 = \lambda_{-w}^n \Pi_{-w} (\psi) (\varphi) +  O\left((1-\delta')^n \norm{\varphi}{\Ccal^1} \norm{\psi}{\Ccal^1}\right).
\end{equation*}
By density, there exists a $\Ccal^1$ function $\psi$ 
such that $\Pi_{-w} (\psi) \neq 0$ in $\Bcal$. Following the construction in~\cite[Lemma~3.8]{DemersZhang:2014}, 
and noticing that the test function can be chosen smooth (for instance by mollification), 
we get a function $\varphi \in \Ccal^1$ such that $\Pi_{-w} (\psi) (\varphi) \neq 0$.

\smallskip

In addition, for the functions $\varphi$ and $\psi$ constructed above,
\begin{align*}
 \int_M \varphi \cdot P_{-w}^n (\psi) \dd \mu
 & = \int_M (P_{-w}^*)^n (\varphi) \cdot \psi \dd \mu \\
 & = \int_M \widetilde{P}_w^n (\varphi) \cdot \psi \dd \mu \\
 & = \lambda_w^n \widetilde{\Pi}_w (\psi) (\varphi) + O\left((1-\delta')^n \norm{\varphi}{\Ccal^1} \norm{\psi}{\Ccal^1}\right),
\end{align*}
where $\widetilde{\Pi}_w$ is the eigenprojection of $\widetilde{P}_w$ corresponding to the eigenvalue 
$\lambda_w$. Hence, $\Pi_{-w} (\psi) (\varphi) = \widetilde{\Pi}_w (\varphi) (\psi) \neq 0$ and $\lambda_w = \lambda_{-w}$.

\smallskip

The function $w \mapsto \lambda_w$ is even on a neighborhood of zero, 
and since $\lambda_w = \lambda_{-w} = \overline{\lambda}_w$, it is also real-valued.
As a consequence, for the same value $\delta >0$, for all $\Lambda$:
\begin{equation*}
 \Sp (P_\Lambda \acts \Bcal_\Lambda) 
 \subset \overline{B}(0,1-\delta) \cup [1-\delta, 1].
\end{equation*}
This finishes the proof of Theorem~\ref{thm:ExistenceResonances}.

\begin{remark}[Real eigendistributions]
 
 The eigenvalues $\lambda_w \in [1-\delta, 1]$ are real, and the corresponding eigenspaces are $2$-dimensional. 
 These eigenspaces are spanned by pairs of eigendistributions $\{h_{\Lambda,\chi}^\Re, h_{\Lambda,\chi}^\Im\}$, 
 which are real (in that $h_{\Lambda,w}^\Re (\varphi)$ and $h_{\Lambda,w}^\Im (\varphi)$ are both real 
 for any real test function $\varphi$) and can be chosen as:
 \begin{equation*}
 \begin{array}{rcl}
  h_{\Lambda,\chi}^\Re & := &  \Re (h_w \otimes \chi) \\
  h_{\Lambda,\chi}^\Im & := &  \Im (h_w \otimes \chi),
 \end{array} 
 \end{equation*}
 where $\chi = e^{i \langle w, \cdot \rangle}$ and $h_w$ is an eigendistribution 
 for $P_w$. 
\end{remark}

\subsection{Distribution of the resonances}
\label{subsec:DistributionResonances}

Now, we shall discuss the convergence of the spectral densities and 
prove Propositions~\ref{prop:CompactenessSpectralMeasures}, \ref{prop:DistributionResonancesDim1} 
and~\ref{prop:DistributionResonancesDim2}.

\begin{proof}[Proof of Proposition~\ref{prop:CompactenessSpectralMeasures}]

Let $\rho_0$ be the upper bound on $\rho_{\ess} (P \acts \Bcal)$ given by~\cite[Proposition~2.3]{DemersZhang:2012}. 
Since $w \mapsto (P_w)_{w \in 2 \pi \Tbb^2}$ 
is a continuous family of transfer operators with $\rho_{\ess} (P_w \acts \Bcal) \leq \rho_0$, 
by continuity of the spectrum, for any compact $K \subset \overline{B}(0,\rho_0)^c$, the 
function $w \mapsto \nu_w (K)$ is bounded, where:
\begin{equation*}
 \nu_w 
 = \sum_{\substack{\lambda \text{ resonance of } P_w \\ |\lambda| > 1-\rho_0}} \delta_\lambda,
\end{equation*}
and where the resonances are counted with multiplicity.

\smallskip

The spectral decomposition yields, for any rank $2$ lattice $\Lambda$,
\begin{equation*}
 \nu_\Lambda 
 = \frac{1}{|G|} \sum_{w \in \widehat{G}} \nu_w,
\end{equation*}
whence $\Lambda \mapsto \nu_\Lambda (K)$ is also bounded uniformly in $\Lambda$. 
Compactness of subprobability measures on $K$ yields the existence of limit distributions 
of $(\nu_{\Lambda_N})$ on any compact $K$, and a diagonal argument yields 
Proposition~\ref{prop:CompactenessSpectralMeasures}.
\end{proof}

Propositions~\ref{prop:DistributionResonancesDim1} and~\ref{prop:DistributionResonancesDim2} 
follow from the discussion after~\cite[Theorem~3.1]{JakobsonNaudSoares:2019}, with 
some care to make the constants explicit.

\begin{proof}[Proof of Proposition~\ref{prop:DistributionResonancesDim1}]

Recall that $\Lambda_N^{(1)} = N \Z \otimes \Z$. By the aforementioned discussion, there exists 
$\delta_0 >0$ such that, for all $f \in \Ccal (\R, \C)$ supported on $[1-\delta_0,1]$,
\begin{equation*}
 \lim_{N \to + \infty} \int_{1-\delta_0}^1 f \dd \nu_N^{(1)} 
 = \frac{1}{2\pi} \int_\R f(\lambda_{t e_1}) \dd t,
\end{equation*}
with $e_1 = (1,0)$. The constant $2\pi$ comes from the different parametrization we use for $t$. 
By the Morse lemma, there exists a $\Ccal^1$ diffeomorphism of the real line $\Psi$ such that 
$\Psi' (0) = 1$ and:
\begin{equation*}
 \frac{1}{2\pi} \int_\R f(\lambda_{t e_1}) \dd t 
 = \frac{1}{2\pi} \int_\R f \left(1-\frac{\Sigma_{11} t^2}{2}\right) |\Psi' (t)| \dd t.
\end{equation*}
Let $\varepsilon >0$. If $f$ is supported on a small enough neighborhood of $1$, 
\begin{equation*}
 \left| \frac{1}{2\pi} \int_\R f(\lambda_{t e_1}) \dd t - \frac{1}{2\pi} \int_\R f \left(1-\frac{\Sigma_{11} t^2}{2}\right) \dd t \right| 
 \leq \varepsilon \norm{f}{\infty},
\end{equation*}
and:
\begin{equation*}
 \frac{1}{2\pi} \int_\R f \left(1-\frac{\Sigma_{11} t^2}{2}\right) \dd t 
 = \frac{1}{\pi} \int_0^{+ \infty} f \left(1-\frac{\Sigma_{11} t^2}{2}\right) \dd t 
 = \int_0^{+ \infty} f (1-u) \frac{1}{\pi\sqrt{2 \Sigma_{11} u}} \dd u,
\end{equation*}
finishing the proof of Proposition~\ref{prop:DistributionResonancesDim1}. 
\end{proof}

Proposition~\ref{prop:DistributionResonancesDim2} follows from the same kind of computations.

\end{document}